\documentclass{article}
\usepackage[latin1]{inputenc}
\usepackage{amsmath}
\usepackage{amsthm,amssymb,amsfonts}
\usepackage{graphicx}
\usepackage{verbatim}
\usepackage{color,cite}
\usepackage{latexsym}
\usepackage{lscape}
\usepackage[T1]{fontenc}
\usepackage[all,cmtip]{xy}
\usepackage{caption}
\usepackage{xspace}
\usepackage{enumitem}
\usepackage{tikz}
\usepackage{hyperref}
\usepackage{xcolor}
\usepackage{tikz-cd}
\usepackage{slashbox}
\usetikzlibrary{patterns}
\usetikzlibrary{shapes,arrows,spy,positioning}

\usepackage[section]{placeins}

\theoremstyle{plain}
\newtheorem{theorem}{Theorem}[section]
\newtheorem{lemma}[theorem]{Lemma}
\newtheorem{proposition}[theorem]{Proposition}
\newtheorem{corollary}[theorem]{Corollary}

\newtheorem*{definition}{Definition}
\newtheorem{example}[theorem]{Example}
\newtheorem*{warning}{Warning}

\newtheorem*{convention}{Convention}
\newtheorem*{remark}{Remark}

\title{A New Approach to the Automorphism Group of a Platonic Surface}
\author{David Aulicino\thanks{The author was partially supported by the National Science Foundation under Award Nos. DMS - 1738381, DMS - 1600360, and PSC-CUNY Grants $\sharp$~60571-00 48 and $\sharp$~61639-00 49.}}
\date{}

\begin{document}

\newcommand{\splin}{\text{SL}_2(\mathbb{R})}
\newcommand{\spolin}{\text{SO}_2(\mathbb{R})}
\newcommand{\nc}{\newcommand}

\nc\bB{\mathbb{B}}
\nc\bC{\mathbb{C}}
\nc\bD{\mathbb{D}}
\nc\bE{\mathbb{E}}
\nc\bF{\mathbb{F}}
\nc\bG{\mathbb{G}}
\nc\bH{\mathbb{H}}
\nc\bI{\mathbb{I}}
\nc{\bJ}{\mathbb{J}}
\nc\bK{\mathbb{K}}
\nc\bL{\mathbb{L}}
\nc\bM{\mathbb{M}}
\nc\bN{\mathbb{N}}
\nc\bO{\mathbb{O}}
\nc\bP{\mathbb{P}}
\nc\bQ{\mathbb{Q}}
\nc\bR{\mathbb{R}}
\nc\bS{\mathbb{S}}
\nc\bT{\mathbb{T}}
\nc\bU{\mathbb{U}}
\nc\bV{\mathbb{V}}
\nc\bW{\mathbb{W}}
\nc\bY{\mathbb{Y}}
\nc\bX{\mathbb{X}}
\nc\bZ{\mathbb{Z}}

\nc\cA{\mathcal{A}}
\nc\cB{\mathcal{B}}
\nc\cC{\mathcal{C}}
\nc\cD{\mathcal{D}}
\nc\cE{\mathcal{E}}
\nc\cF{\mathcal{F}}
\nc\cG{\mathcal{G}}
\nc\cH{\mathcal{H}}
\nc\cI{\mathcal{I}}
\nc{\cJ}{\mathcal{J}}
\nc\cK{\mathcal{K}}
\nc\cM{\mathcal{M}}
\nc\cN{\mathcal{N}}
\nc\cO{\mathcal{O}}
\nc\cP{\mathcal{P}}
\nc\cQ{\mathcal{Q}}
\nc\cS{\mathcal{S}}
\nc\cT{\mathcal{T}}
\nc\cU{\mathcal{U}}
\nc\cV{\mathcal{V}}
\nc\cW{\mathcal{W}}
\nc\cY{\mathcal{Y}}
\nc\cX{\mathcal{X}}
\nc\cZ{\mathcal{Z}}

\maketitle

\begin{abstract}
We borrow a classical construction from the study of rational billiards in dynamical systems known as the ``unfolding construction'' and show that it can be used to study the automorphism group of a Platonic surface.  More precisely, the monodromy group, or deck group in this case, associated to the cover of a regular polygon or double polygon by the unfolded Platonic surface yields a normal subgroup of the rotation group of the Platonic surface.  The quotient of this rotation group by the normal subgroup is always a cyclic group, where explicit bounds on the order of the cyclic group can be given entirely in terms of the Schl\"afli symbol of the Platonic surface.  As a consequence, we provide a new derivation of the rotation groups of the dodecahedron and the Bolza surface.
\end{abstract}


\section{Introduction}
\label{IntroSection}

Platonic surfaces are natural generalizations to higher genus of the surfaces of the classical Platonic solids.  Consider a collection of regular (Euclidean) $p$-gons with topological identifications among pairs of edges to form an orientable closed surface with the property that the automorphism (Euclidean isometry) group acts transitively on flags of faces, edges, and vertices.  The skeleton of a Platonic surface is commonly called a regular map.  If the degree of a vertex in the regular map is $q$, then the regular map has Schl\"afli symbol $\{p,q\}$.  By regarding the regular $p$-gons with their Euclidean geometry, elementary concepts from the study of translation surfaces can be brought to bear on the subject of Platonic surfaces.  In particular, we consider the deck group from topology of the unfolded Platonic surface induced by a covering map $\pi_{\Pi}$ to the $p$-gon or double $p$-gon depending on the parity of $p$.  Precise definitions will be given in the next section.  The rotation group of a Platonic surface is the index two subgroup of the automorphism group consisting of \emph{orientation preserving} automorphisms.  
 
\begin{theorem}
\label{PartNormTower}
The deck group $\text{Mon}(\pi_{\Pi})$ is isomorphic to a normal subgroup $N$ of the rotation group\footnote{The rotation group is often called the \emph{rotary group}.} $\text{Rot}(D)$ of the Platonic surface $D$.  Furthermore,
$$\text{Rot}(D)/N$$
is a cyclic group.
\end{theorem}

We use the notation $\text{Mon}(\pi_{\Pi})$ for the deck group because the real object we consider is the monodromy group from algebraic topology.  The monodromy group refers to a \emph{different object} in the field of regular maps and we avoid the terminology altogether in this introduction.  Nevertheless, we use it in the remainder of the paper.

For a precise version of this theorem with additional information see Theorem \ref{MonGrpNormSubgrpRot}.  The significance of the theorem is that given a Platonic surface $D$, $\text{Mon}(\pi_{\Pi})$ is very easy to compute as a subgroup of a large symmetric group.  While the theorem does not always provide sufficient information to compute the automorphism group of the regular map, it sometimes suffices to completely determine $\text{Rot}(D)$.

The order of the cyclic group in Theorem \ref{PartNormTower} can be bounded in terms of the Schl\"afli symbol as follows:

\begin{proposition}
\label{Computek}
Let $D$ be a Platonic surface with Schl\"afli symbol $\{p,q\}$.  Define $k'$ to be the smallest positive integer satisfying the equation 
$$k'\frac{q(p-2)}{p}\pi \equiv  0\mod 2\pi.$$
If $k$ is the degree of the covering of the unfolded Platonic surface, then $k' | k$.  If $p$ is even, then 
$$k' \leq k \leq p,$$
and if $p$ is odd, then
$$k' \leq k \leq 2p.$$
\end{proposition}

Proposition \ref{Computek} will be proven in Section \ref{CompSect}.  As a corollary of the results in that section, we get

\begin{theorem}
\label{GCD1Thm}
Let $D$ be a Platonic surface with Schl\"afli symbol $\{p,q\}$.  If $p$ and $q$ are relatively prime and either
\begin{itemize}
\item $p$ or $q$ is divisible by $4$, or
\item $p$ and $q$ are odd,
\end{itemize}
then
$$\text{Mon}(\pi_{\Pi}) \cong \text{Rot}(D).$$
\end{theorem}

This will be used in Section \ref{ExampleSect} to derive the rotation groups of the dodecahedron and the Bolza surface.

The significance of this paper is that it turns the problem of classifying Platonic surfaces or regular maps into a purely geometric problem where the group theoretic aspects should be manageable.

What is most surprising in this context is that the unfolding construction has been used in the field of dynamical systems since the 1930's to study the dynamics of a ball in certain billiard tables.  Here the unfolding construction proves to be so natural that it extracts group theoretic information from the Platonic surface with complete disregard for any dynamical system on the surface.

There have been numerous works studying this problem graph theoretically, group theoretically, and hyperbolic geometrically.  From the graph theoretic perspective, a survey can be found in \cite{ConderSurvey}.  A census of regular maps in small genus was constructed in \cite{ConderDobcsanyiCensus}, and a far larger census can be found on the website of Marsten Conder, where group theoretic methods in Magma were used to generate the list \cite{ConderWebsite}.  From a geometric perspective, works of \cite{KarcherWeber} give an excellent comprehensive exposition of some particular cases, and \cite{McMullenSchulteText} is a recent textbook covering a wealth of material with an extensive bibliography.

The connection to hyperbolic geometry is as follows.  If a Riemann surface has a sufficiently large automorphism group\footnote{Automorphism groups of order strictly greater than $12(g-1)$ suffice.  In the paper, \cite{SwinarskiLargeAutGrp} large is taken to mean order at least $4(g-1)$.  In any case, this will not be used in this paper.}, then its automorphism group is isomorphic to a quotient of a triangle group by a torsion-free, finite-index normal subgroup.  Therefore, the surface is tiled by hyperbolic triangles and covers a sphere consisting of two isometric hyperbolic triangles.  By replacing the hyperbolic triangles by appropriate Euclidean triangles, the surface is tiled by regular $p$-gons and is in fact a Platonic surface.  Its skeleton is a regular map.

We remark here that these objects occur in the study of Bely\u{\i} maps and dessins d'enfants as well.

The key perspective here is that we regard Platonic surfaces by their almost everywhere flat geometry.  For example, locally a cube looks geometrically like the plane everywhere except for its $8$ vertices, which have angle $3\pi/2$ instead of $2\pi$, like the plane.  We observe that the interiors of the edges of the polygons do not see this issue and therefore, are intrinsically geometrically indistinguishable from a small disc in the Euclidean plane.  We encourage the reader to look at Figure \ref{BolzaSurface}, which represents the Bolza surface in genus two with a flat geometric structure imposed on it.  The flat geometry is clear from the fact that the figure is drawn on a piece of paper.  We observe that the negative curvature of the surface is ``hidden'' in the corners of the octagons which have angle $3(3\pi/4) = 9\pi/4$.  Interestingly, the problem is addressed by turning specific Riemann surfaces into so-called flat cone surfaces by choosing a very particular flat structure on the surface that uncovers properties of the hyperbolic geometry.  

\noindent \textbf{Outline:} In Section \ref{PreliminariesSect}, we present the relevant background.  The main theorem is proved in Section \ref{MainProofsSect}.  Section \ref{CompSect} contains estimates on the order of the cyclic group in Theorem \ref{PartNormTower}.  Finally, Section \ref{ExampleSect} gives examples of the main theorem applied to both the dodecahedron and the Bolza surface.

\subsubsection*{Acknowledgments}
The author thanks Pat Hooper for careful reading and assistance with many of the technical proofs in this paper, and Jayadev Athreya for originally communicating the possible connection between their previous joint work and Platonic surfaces.  He also thanks Ferr\'an Valdez for inspiring this work, and for his continued interest in it.  Finally, he is grateful to Chris Judge and David Torres-Teigel for helpful conversations.

\section{Preliminaries}
\label{PreliminariesSect}

In this section, we define all of the necessary terms and record all of the results needed from \cite{AthreyaAulicinoHooperDodecahedron}.

\subsection{Geometry}

\subsubsection{The Primary Objects}

\begin{definition}
Given a collection of regular Euclidean $p$-gons, let $D$ be a surface without boundary formed from an identification of the edges of the polygons so that the result is closed and orientable.  Define a \emph{flag}\footnote{This is also called a \emph{blade} in the graph theory literature.} of $D$ to be the triple $(f, e, v)$ of a polygon $f$ with one of its boundary edges $e$, where $e$ is a vector in $\bR^2$, and a vertex $v$ incident with $e$.  If the automorphism group of $D$ is transitive on flags, then $D$ is called a \emph{Platonic surface}.
\end{definition}

\begin{remark}
For the reader unfamiliar with this concept, we emphasize that this is a topological construction and surfaces of higher genus with curvature as specified above cannot be embedded in $3$-dimensions.  In spite of this, these surfaces are very easy to depict by hiding the points with angle greater than $2\pi$ in the corners of the presentation.  See Figure \ref{BolzaSurface} for an example.
\end{remark}

\begin{definition}
The \emph{Schl\"afli symbol} of a Platonic surface is the data $\{p,q\}$, where every face is bounded by $p$ edges and every vertex has degree $q$.
\end{definition}


\begin{definition}
Given a collection of regular Euclidean $p$-gons, let $M$ be a surface without boundary given by an identification of the edges of the polygons so that the result is closed and orientable, and the angle at every point is an integral multiple of $2\pi$.  We call such a surface with a choice of horizontal direction a \emph{regular $p$-gon-tiled surface}.
\end{definition}

See Figure \ref{BolzaSurfaceUnfold} for an example.  In fact, Figure \ref{fig:pi} described in the example below provides perhaps two of the most elementary examples of such a surface.  We refer the reader to the figures in \cite{AthreyaAulicinoHooperDodecahedron} for many more examples coming from the classical Platonic solids.  We also remark that the regular $p$-gon-tiled surfaces considered in this paper are almost examples of Platonic surfaces.  The subtle difference is that regular $p$-gon-tiled surfaces have a choice of horizontal direction and Platonic surfaces do not.  Forgetting the additional information of the choice of direction, the regular $p$-gon-tiled surfaces considered in this paper will be examples of Platonic surfaces.

\begin{remark}
The choice of horizontal direction is actually a natural piece of information that comes from complex analysis.  Away from the corners of the polygons, a local neighborhood can be viewed as a small neighborhood in the complex plane with the differential $dz$.  Consider the differential $dz = dx + \sqrt{-1}\,dy$ on $\bC$.  The equations $\text{Re}(dz) = 0$ and $\text{Im}(dz) = 0$ extract the families of vertical and horizontal lines, respectively, in the complex plane.  The choice of horizontal, or vertical, direction is exactly specifying how the differential lies on the surface.  Without this, the differential could only be distinguished up to multiplication by a complex unit.
\end{remark}

\begin{example}
The first examples the reader should consider for this paper are known as the $p$-gon or double $p$-gon, depending on the parity of $p$.

Observe that if $p$ is even, every side of a regular $p$-gon has a unique opposite parallel side.  Identifying these sides leads to an example of a translation surface (in fact a regular $p$-gon-tiled surface) denoted here by $\Pi_p$.

The second example to consider is when $p$ is odd.  In this case, take two regular $p$-gons and again observe that every side has a unique opposite parallel side to which it is parallel.  Identifying again in this case leads to a translation surface denoted $\Pi_p$.  (The parity of $p$ implicitly records which one we consider.)
\end{example}

\begin{figure}
\begin{center}
\includegraphics[height=1.5in]{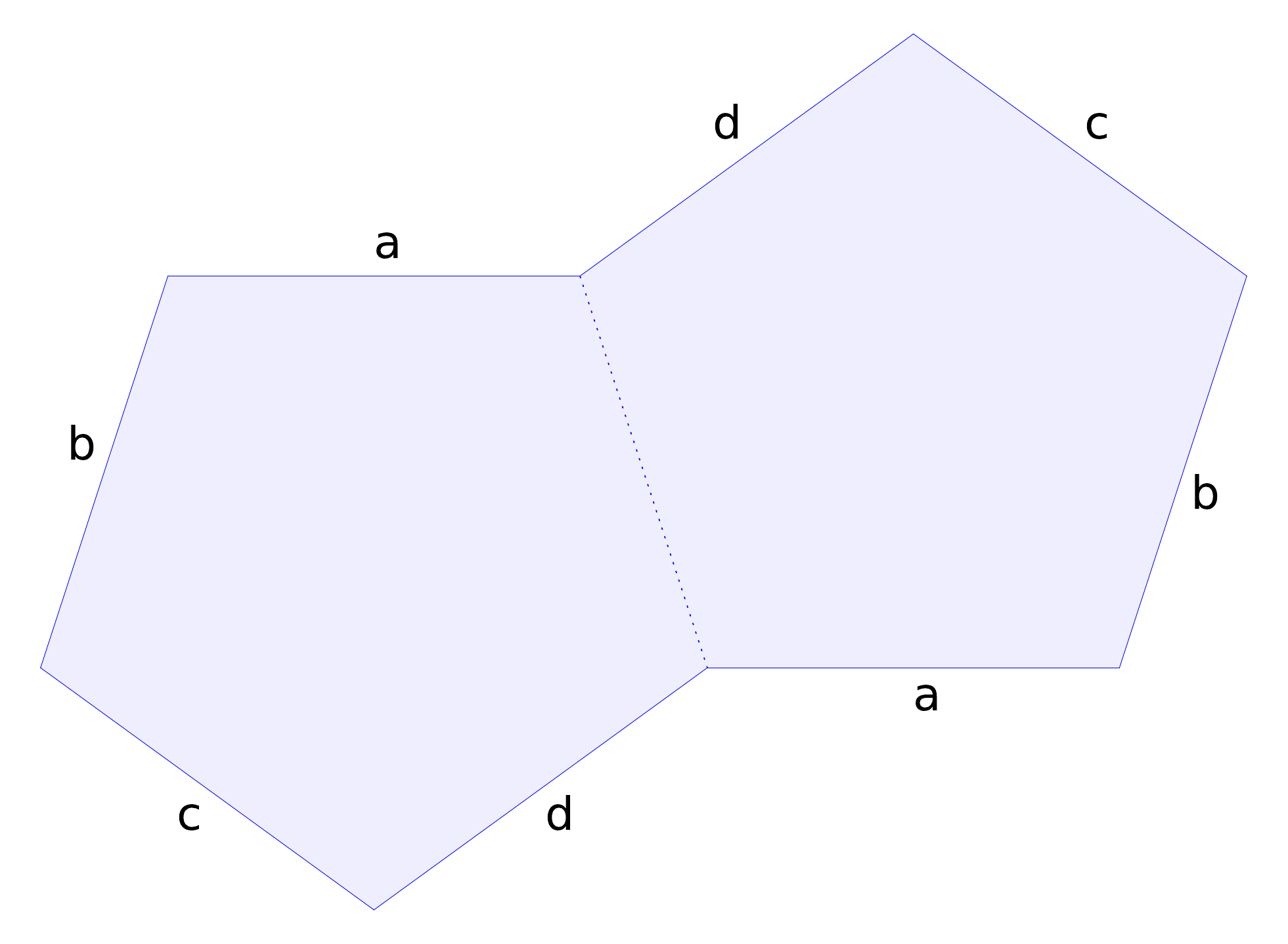}
\hfill
\includegraphics[height=1.5in]{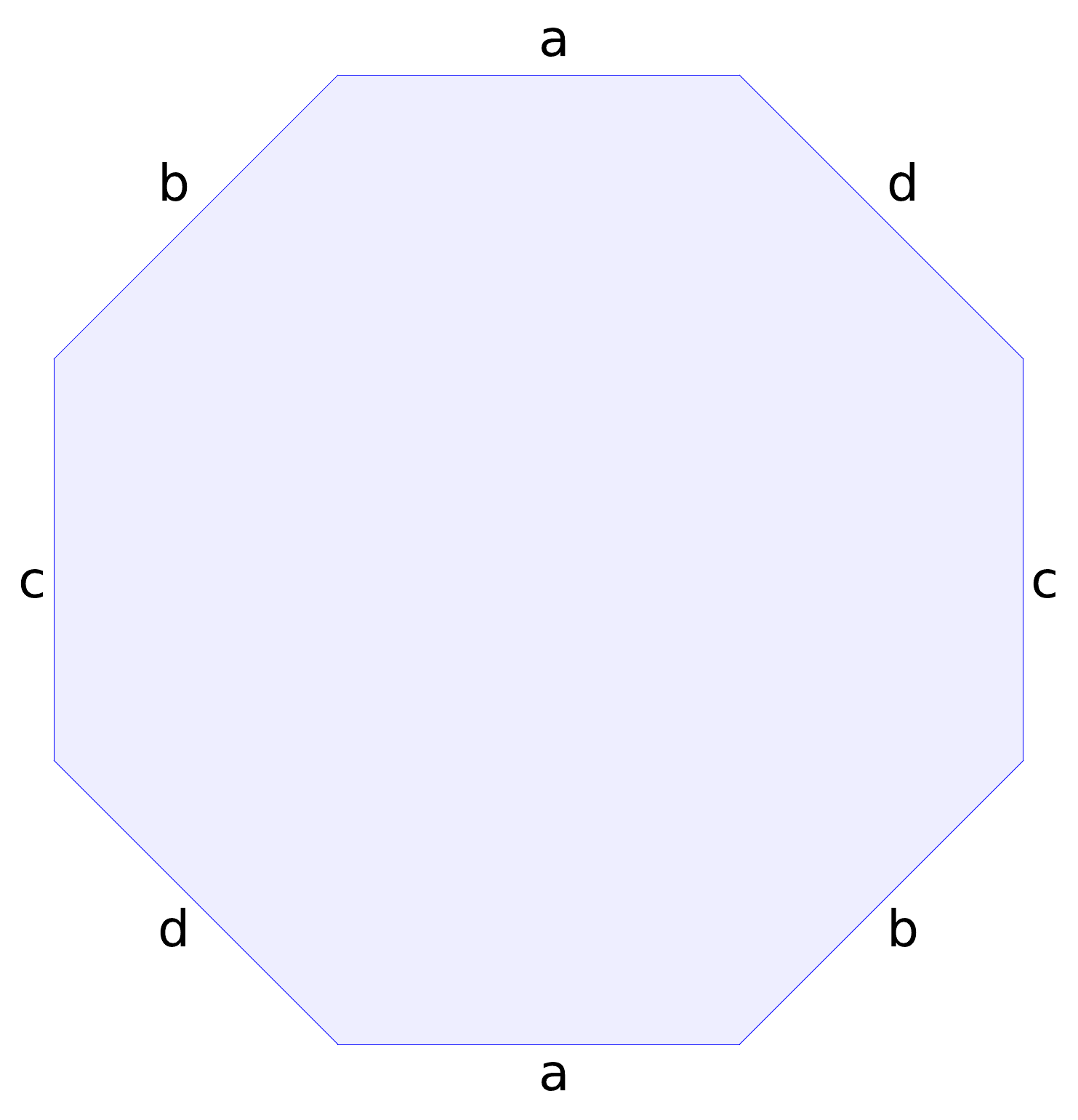}
\end{center}
\caption{The Examples $\Pi_p$: $p = 5$ (left) and $p = 8$ (right) \cite[Fig. 2]{AthreyaAulicinoHooperDodecahedron}}
\label{fig:pi}
\end{figure}

\subsubsection{Unfolding}

In this section we borrow a classical construction from the theory of rational billiards \cite{FoxKershner36} to introduce a natural subgroup of $\text{Aut}(G)$.  We observe that the holonomy on a Platonic surface most often will not lie in $2\pi \bZ$.  Nevertheless, for some $k$, it is always possible to pass to a degree $k$ branched cover of the Platonic surface such that it does.

\begin{remark}
The unfolding construction is also known as the \emph{canonical $k$-cover} \cite[$\S$ 2.1]{BCGGMkDiffs}.
\end{remark}


\begin{definition}
The \emph{holonomy} of a closed curve is the change in angle as we parallel transport a tangent vector around the curve.
\end{definition}

\begin{definition}
Define the \emph{unfolding} $\tilde D$ of a Platonic surface $D$, with branched covering map $\pi_D: \tilde D \rightarrow D$, to be the smallest cover such that the holonomy of the cover lies in $2\pi \bZ$, i.e. every curve on $D$ lifts to a curve on $\tilde D$ with holonomy in $2\pi \bZ$.  Let $k$ denote the degree of $\pi_D$.
\end{definition}

We will always assume that the unfolding is a regular $p$-gon tiled surface by choosing a horizontal direction.  This is exactly the same concept as choosing a branch cut of a Riemann surface in complex analysis.

This exists and is discussed in more detail in \cite[$\S$ 2.2]{AthreyaAulicinoHooperDodecahedron}.  The idea is that by rotating a presentation of a Platonic surface appropriately, every edge is rotated into $k$ copies with a unique orientation so that each edge has a unique \emph{parallel} copy to which it can be identified.  This is exactly what is depicted in Figure \ref{BolzaSurfaceUnfold}.  An unfolding of a Platonic surface is a regular $p$-gon-tiled surface.

We observe that there is a natural branched covering $\pi_{\Pi}: \tilde D \rightarrow \Pi_p$ that actually possesses extra structure.  The covering map $\pi_{\Pi}$ induces a bijection on the set of unit vectors at a point away from the vertices of the Platonic surface.  In other words directions are well-defined modulo $2\pi$ when passing between $\tilde D$ and $\Pi_p$.  Observe that this is \emph{false} for $\pi_D$ if $k > 1$.

\subsection{Groups}

\begin{definition}
Let $D$ be a Platonic surface with automorphism group $\text{Aut}(D)$.  The \emph{rotation subgroup} $\text{Rot}(D)$ of $\text{Aut}(D)$ is the subgroup of orientation preserving automorphisms.
\end{definition}

Since $\text{Rot}(D)$ is a subgroup of $\text{Aut}(D)$ of index two, it is a normal subgroup of $\text{Aut}(D)$.

\begin{definition}
A \emph{pair}\footnote{This is dual to the concept of a dart, which is a pair $(e,v)$.} is a tuple $(f,e)$.
\end{definition}

We will primarily be concerned with the group $\text{Rot}(D)$.  Consequentially, it suffices to reduce flags to pairs because rotations are orientation preserving and the vertex would provide redundant information.

Let $p: Y \rightarrow X$ be a branched covering map of degree $n$.  There is a natural representation of the fundamental group into the symmetric group known as the monodromy representation.  The \emph{monodromy representation}
$$\psi_p : \pi_1(X, p_0) \rightarrow S_n$$
is given by numbering the sheets of the cover $Y$ by $\{1, \ldots, n\}$ and recording the induced permutation of the sheets by lifting loops in $\pi_1(X, p_0)$ to $\pi_1(Y, \tilde p_0)$.  

\begin{warning}
There is a concept known as the monodromy group in the study of regular maps.  The monodromy group defined here is \emph{not} that object.  The monodromy group here is the one commonly found in topology \cite{BredonTopologyAndGeometry}.
\end{warning}

\begin{definition}
The image of the monodromy representation $\psi_{\pi_{\Pi}}$ is the \emph{monodromy group} $\text{Mon}(\pi_{\Pi})$.
\end{definition}

In particular, the monodromy group associated to the covering map $\pi_{\Pi}$ will be of particular interest.  We summarize the two covering maps that are the subject of this paper in the following diagram.

\begin{equation}
\label{eq:diagram}
\begin{tikzcd}
 & \tilde D \arrow{dr}{\pi_{\Pi}} \arrow{dl}[swap]{\pi_D}\\
D  & & \Pi_p
\end{tikzcd}
\end{equation}

\subsection{Previous Results}

The key theorem from \cite{AthreyaAulicinoHooperDodecahedron} that serves as the starting point for this paper is the following.

\begin{theorem}[\cite{AthreyaAulicinoHooperDodecahedron}, Thm. 4.5]
\label{NormCov}
The covering $\pi_{\Pi}: \tilde D \rightarrow \Pi_p$ is normal.  Furthermore, $\tilde D$ is a Platonic surface, i.e. the group of orientation preserving isometries of $\tilde D$, denoted $\text{Rot}(\tilde D)$, is transitive on pairs.
\end{theorem}

\section{Proofs of the Main Results}
\label{MainProofsSect}

If $p: Y \rightarrow X$, then let $\text{Deck}(p)$ denote the group of deck transformations of $Y$ with respect to the covering map $p$.



\begin{lemma}
\label{DiagCommutes}
Let $\pi: Y \rightarrow X$ be a normal covering of Platonic surfaces $Y$ and $X$.  Given an orientation preserving isometry $F: X \rightarrow X$, there exists an automorphism $\tilde F: Y \rightarrow Y$ such that the diagram below commutes.
\end{lemma}

\begin{center}
\begin{tikzcd}
Y \arrow[r, dashed, "\tilde F"] \arrow[d, "\pi"]
& Y \arrow[d, "\pi"] \\
X \arrow[r, "F"]
& X
\end{tikzcd}
\end{center}




\begin{proof}
Given $F: X \rightarrow X$, choose a pair $(P_0, e_0)$ on $X$ and its image $(P_1, e_1)$ under $F$.  Choose lifts of each pair on $Y$ so that $(P_i, e_i)$ lifts to $(\tilde P_i, \tilde e_i)$ for $i \in \{0,1\}$.  Since $Y$ is a Platonic surface, there is an isometry $\tilde F$ such that $\tilde F(\tilde P_0, \tilde e_0) = (\tilde P_1, \tilde e_1)$.  Therefore, $\tilde F$ is a lift of $F$.
\end{proof}

\begin{proposition}
\label{CovMapIndSurjHomomor}
Let $\pi: Y \rightarrow X$ be a normal covering of Platonic surfaces $Y$ and $X$.  The covering map $\pi$ induces a surjective homomorphism
$$\phi_{\pi}: \text{Rot}(Y) \rightarrow \text{Rot}(X),$$
and
$$\text{ker}(\phi_{\pi}) = \text{Deck}(\pi) \cong \text{Mon}(\pi).$$
\end{proposition}

\begin{proof}
We construct the map $\phi_{\pi}$ explicitly.  Given $\tilde F \in \text{Rot}(Y)$, choose a pair $(\tilde f_1, \tilde e_1)$ on $Y$ and let $\tilde F(\tilde f_1, \tilde e_1) = (\tilde f_2, \tilde e_2)$.  We abuse notation and write $\pi(f,e)$ to mean $(\pi(f), \pi(e))$.  Now consider the pairs $(f_i, e_i) = \pi(\tilde f_i, \tilde e_i)$, for $i = 1,2$.  This defines a unique rotation of $X$ by letting $F$ be the unique rotation with the property that $F(f_1, e_1) = (f_2, e_2)$.  Use Lemma \ref{DiagCommutes} to lift $F$ to a rotation $\tilde F_0 \in \text{Rot}(Y)$.  For convenience, let $\Delta = \text{Deck}(\pi)$.  Furthermore, $\tilde F_0$ maps $(\tilde f_1', \tilde e_1')$ to $(\tilde f_2', \tilde e_2')$, where $(\tilde f_i', \tilde e_i')$ is a lift of $(f_i, e_i)$ for $i = 1,2$.  Then because the lift is unique up to action by the deck group, and the deck group is transitive because $\pi$ is normal, for $i = 1,2$ there exist elements $\delta_i \in \Delta$ such that $\delta_i(\tilde f_i', \tilde e_i') = (\tilde f_i, \tilde e_i)$.  Therefore, the map $\delta_2 \circ \tilde F_0 \circ \delta_1$ is a rotation of $Y$ that sends $\tilde F(\tilde f_1, \tilde e_1) = (\tilde f_2, \tilde e_2)$.  Since the action of a rotation on a pair uniquely determines the rotation, we have $\delta_2 \circ \tilde F_0 \circ \delta_1 = \tilde F$.  This implies that the construction of $F$ given $\tilde F$ did not depend on the initial choice of pair.  Therefore, $\phi_{\pi}$ is a well-defined map.  The fact that $\phi_{\pi}$ is a homomorphism follows from the following composition of commutative diagrams.
\begin{center}
\begin{tikzcd}
Y \arrow[r, "\tilde F_1"] \arrow[d, "\pi"] 
& Y \arrow[r, "\tilde F_2"] \arrow[d, "\pi"] 
& Y \arrow[d, "\pi"] \\
X \arrow[r, "F_1"] & X \arrow[r, "F_2"] & X
\end{tikzcd}
\end{center}

Finally, the kernel of $\phi_{\pi}$ consists of exactly those maps such that $\pi(f_1, e_1) = \pi(f_2, e_2)$.  However, this is exactly the condition that two maps differ by a deck transformation.  Surjectivity follows from considering lifts of all possible pairs of pairs $(f,e)$ on $X$.
\end{proof}

Observe that Theorem \ref{CovMapIndSurjHomomor} applies to $\phi_{\Pi}$ because $\Pi_p$ is indeed a Platonic surface and by Theorem \ref{NormCov}, $\pi_{\Pi}$ is a normal cover and $\tilde D$ is a Platonic surface.  Theorem \ref{CovMapIndSurjHomomor} also applies to $\phi_D$ because $\tilde D$ is a Platonic surface by Theorem \ref{NormCov}.

\begin{lemma}
\label{GamMonSubgrpGamRotTildeD}
There is an injective homomorphism
$$\tilde \Psi: \text{Mon}\left(\pi_{\Pi}\right) \hookrightarrow \text{Rot}(\tilde D).$$
\end{lemma}

\begin{proof}
Since $\pi_{\Pi}$ is a normal covering by Theorem \ref{NormCov}, it follows from \cite[Prop. 2.10]{DydakFundGrpsNewApp} or \cite[Thm. 7.2]{MasseyBasicAlgTop} that $\text{Mon}(\pi_{\Pi})$ is isomorphic to $\text{Deck}(\pi_{\Pi})$.  However, the deck transformations of $\pi_{\Pi}$ are naturally isometries of $\tilde D$ because they are homeomorphisms of $\tilde D$ that interchange sheets of $\pi_{\Pi}$ by definition.  Hence, they form a subgroup of $\text{Rot}(\tilde D)$.
\end{proof}

\begin{lemma}
\label{KerPhiDRotEdges}
Let $\delta \in \text{ker}(\phi_D)$ such that for some pair $(f,e)$, $\delta$ sends $(f,e)$ to $(f', e')$.  Let $\theta$ be the angle between $e$ and $e'$ on $\tilde D$.  Then $\theta \equiv 0 \mod 2\pi$ if and only if $\delta = \text{Id}$.
\end{lemma}

\begin{proof}
This is trivial if $\delta = \text{Id}$.  Observe that $\text{ker}(\phi_D)$ is exactly the deck group of the degree $k$ covering $\pi_D: \tilde D \rightarrow D$.  By definition, this covering consists of all rotations of $D$ by angle $2\pi/k$.  Hence, only the trivial element of the deck group does not rotate edges.
\end{proof}

The main theorem now follows from elementary results from finite group theory.  The reader is advised to consult the following diagram throughout the proof.

\begin{equation*}
\begin{tikzcd}[column sep=small]
& \text{Mon}(\pi_{\Pi}) \arrow[ddl, dashed, hook, bend right,
"\Psi" '] \arrow[d, hook, "\tilde{\Psi}"] & \\
 & \text{Rot}(\tilde{D}) \arrow[dl, two heads, "\phi_D" ' ]
 \arrow[dr, two heads, "\phi_{\Pi}" ] & \\
 \text{Rot}(D) & & \text{Rot}(\Pi)
\end{tikzcd}
\end{equation*}

\begin{theorem}
\label{MonGrpNormSubgrpRot}
The monodromy group $\text{Mon}(\pi_{\Pi})$ is isomorphic to a normal subgroup of $\text{Rot}(D)$.  Furthermore, the group
$$\text{Rot}(D)/\phi_D\left(\text{Mon}(\pi_{\Pi})\right)$$
is cyclic.
\end{theorem}

\begin{proof}
By Lemma \ref{GamMonSubgrpGamRotTildeD}, there is an injective homomorphism 
$$\tilde \Psi: \text{Mon}(\pi_{\Pi}) \hookrightarrow \text{Rot}(\tilde D).$$
By Proposition \ref{CovMapIndSurjHomomor} and the remarks following it, $\tilde \Psi(\text{Mon}(\pi_{\Pi}))$ is a normal subgroup of $\text{Rot}(\tilde D)$.  Therefore, $\phi_D\left(\tilde \Psi(\text{Mon}(\pi_{\Pi}))\right)$ is a normal subgroup of $\phi_D(\text{Rot}(\tilde D))$.  By Proposition \ref{CovMapIndSurjHomomor} and the remarks following it, $\phi_D$ is surjective, so $\phi_D\left(\tilde \Psi(\text{Mon}(\pi_{\Pi}))\right)$ is a normal subgroup of $\text{Rot}(D) = \phi_D(\text{Rot}(\tilde D))$.  Clearly,
$$\text{Rot}(D) \cong \text{Rot}(\tilde D)/\ker(\phi_D).$$
Consider the induced homomorphism onto the quotient group
$$\Psi: \text{Mon}(\pi_{\Pi}) \rightarrow \text{Rot}(\tilde D)/\ker(\phi_D).$$
Since every element in the image of $\Psi$ preserves the angles between edges by Lemma \ref{KerPhiDRotEdges} and every non-trivial element in $\ker(\phi_D)$ changes the angle of edges, $\Psi$ is an injective homomorphism.

We claim that $\text{Mon}(\pi_{\Pi})$ is the kernel of the derivative map into $O(2)$.  Indeed $\text{Mon}(\pi_{\Pi})$ consists of all permutations of the sheets, each isomorphic to $\Pi_p$, contained in $\text{Rot}(\tilde D)$ that do not rotate the faces.  Since these permutations correspond to the transformation of one face into another, the derivative of this transformation records the amount of rotation.  Hence, $\text{Mon}(\pi_{\Pi})$ is the kernel of the derivative map as claimed.  The cyclic claim follows because the derivative map has finite image in $O(2)$.
\end{proof}



\section{Bounds on Quantities}
\label{CompSect}

Though explicit values for $k$ are hard to determine because they depend on the global geometry of the surface, bounds can be derived in terms of the Schl\"afli symbol.  This in turn provides bounds on algorithms that perform group computations.

\begin{convention}
Throughout the remainder of the paper we set
$$d = \gcd(p,q).$$
\end{convention}

\begin{convention}
Throughout the remainder of the paper $m$ will refer to the number of faces of $D$.
\end{convention}

\begin{proof}[Proof of Prop. \ref{Computek}]
By definition of a Platonic surface, the cone angle at each vertex is equal to
$$\frac{q(p-2)}{p}\pi.$$
Therefore, the value of $k'$ is clearly the smallest quantity necessary to guarantee that the angle at every cone point on the unfolding is an integral multiple of $2\pi$.  Moreover, any positive integer that does not satisfy this equation cannot result in a cone point with cone angle that is a multiple of $2\pi$.  Therefore, the divisibility claim follows.

We examine the holonomy of a closed curve on $D$.  Since $D$ is constructed from regular $p$-gons, the interior angle on each face is $(p-2)\pi/p$.  Therefore, each time a curve enters and exits a face, the angle of a tangent vector can only change by a multiple of $\pi/p$.  Hence, every curve on $D$ has holonomy in $(\pi/p)\bZ$.  Taking an appropriate cyclic $2p$-cover allows the resulting holonomy of every lifted curve to lie in $2\pi\bZ$.

This can be improved under the assumption that $p$ is even.  Observe that $2$ can be factored out of $p-2$ in the numerator.  Therefore, each curve traversing each polygon changes the holonomy by an angle that is a multiple of $2\pi/p$.  Hence, $p$ suffices for the maximal degree of the cover.
\end{proof}




\begin{lemma}
\label{ValueOfk}
Let $D$ have Schl\"afli symbol $\{p,q\}$.  Then the value of $k'$ from Proposition \ref{Computek} is given in Table \ref{ValsofkTable}.
\end{lemma}

\begin{proof}
A $p$-gon has interior angle $\frac{p-2}{p}\pi$, which implies that the total angle at each vertex on the Platonic surface is $\frac{q(p-2)}{p}\pi$.  Consider
$$\frac{p}{d} \left( \frac{q(p-2)}{p} \right) = \frac{q}{d}(p-2) = L,$$
and observe that it is always an integer because $d | q$.  If $q$ is even, then $L$ is even as well.  On the other hand, if $q$ is odd, then there are several cases.  If $p$ is also odd, then $L$ is odd, and $k' = 2p/d$.  If $4|p$, then $2|(p-2)$ and $4 \not | (p-2)$, which implies that $k' = p/d$.  Finally, if $4|(p-2)$, then $k' = p/(2d)$ because $L/2$ is even in this case and $2 \not| d$ because $q$ is odd.
\end{proof}

\begin{table}
\centering
\begin{tabular}{c|ccc}
 \backslashbox{$p$}{$q$} & odd & $0 \mod 4$ & $2 \mod 4$\\
\hline
odd & $\frac{2p}{d}$ & $\frac{p}{d}$ & $\frac{p}{d}$\\
$0 \mod 4$ & $\frac{p}{d}$ & $\frac{p}{d}$ & $\frac{p}{d}$\\
$2 \mod 4$ & $\frac{p}{2d}$ & $\frac{p}{d}$ & $\frac{p}{d}$\\
\end{tabular}
\caption{Values of $k'$ depending on $\{p,q\}$}
\label{ValsofkTable}
\end{table}

The following is well-known and can be found in \cite[Ch. III, Cor. 5.2]{BredonTopologyAndGeometry}.

\begin{proposition}
\label{MonGrpOrderEqDeg}
The order of the monodromy group $\text{Mon}(\pi_{\Pi})$ is equal to the degree of the cover $\pi_{\Pi}: \tilde D \rightarrow \Pi_p$.
\end{proposition}

\begin{proposition}
\label{RotGrpOrder}
Let $D$ be a Platonic surface with Schl\"afli symbol $\{p,q\}$.  The order of the rotation group $\text{Rot}(D)$ is $mp$.
\end{proposition}

\begin{proof}
Clearly, there are $mp$ pairs.  Since the rotation group of the Platonic surface is transitive on pairs by definition of the Platonic surface and specifying how one pair maps to another uniquely determines the map, the order of the rotation group is $mp$.
\end{proof}

\begin{table}
\centering
\begin{tabular}{c|ccc}
 \backslashbox{$p$}{$q$} & odd & $0 \mod 4$ & $2 \mod 4$\\
\hline
odd & $\frac{mp}{d}$ & $\frac{mp}{2d}$ & $\frac{mp}{2d}$\\
$0 \mod 4$ & $\frac{mp}{d}$ & $\frac{mp}{d}$ & $\frac{mp}{d}$\\
$2 \mod 4$ & $\frac{mp}{2d}$ & $\frac{mp}{d}$ & $\frac{mp}{d}$\\
\end{tabular}
\caption{Lower bounds for the order of $\text{Mon}(\pi_{\Pi})$}
\label{OrdersOfMonGrpTable}
\end{table}

\begin{proposition}
\label{MonGrpOrder}
Let $D$ be a Platonic surface with Schl\"afli symbol $\{p,q\}$.  For each $\{p,q\}$, a lower bound for the order of the monodromy group is given in Table \ref{OrdersOfMonGrpTable}.
\end{proposition}

\begin{proof}
By Proposition \ref{MonGrpOrderEqDeg}, the monodromy group has order $km$, when $p$ is even, and $km/2$ when $p$ is odd.  By Proposition \ref{Computek}, $k'$ divides $k$ and by Lemma \ref{ValueOfk}, $k'$ is given in Table \ref{ValsofkTable}.
\end{proof}

\begin{remark}
We observe that all values in Table \ref{OrdersOfMonGrpTable} are indeed integers because if $p$ is odd, then $\Pi_p$ consists of two polygons and $mk$ is indeed an even number.
\end{remark}

\begin{corollary}
\label{IndexMonInRotGrp}
Let $D$ be a Platonic surface with Schl\"afli symbol $\{p,q\}$.  For each $\{p,q\}$, an upper bound for the order of the cyclic group $\text{Rot}(D)/\phi_D\left(\text{Mon}(\pi_{\Pi})\right)$ is given in Table \ref{ValsOfMonGrpInd}.
\end{corollary}

\begin{table}
\centering
\begin{tabular}{c|ccc}
 \backslashbox{$p$}{$q$} & odd & $0 \mod 4$ & $2 \mod 4$\\
\hline
odd & $d$ & $2d$ & $2d$\\
$0 \mod 4$ & $d$ & $d$ & $d$\\
$2 \mod 4$ & $2d$ & $d$ & $d$\\
\end{tabular}
\caption{Upper bounds for the order of $\text{Rot}(D)/\phi_D\left(\text{Mon}(\pi_{\Pi})\right)$}
\label{ValsOfMonGrpInd}
\end{table}

\begin{proof}[Pf. of Thm. \ref{GCD1Thm}]
By inspection of Table \ref{ValsOfMonGrpInd}, we focus on the entries equal to $d$.  Furthermore, at least one of $p$ and $q$ must be odd, otherwise $\gcd(p,q) \geq 2$.  For an entry in Table \ref{ValsOfMonGrpInd} to equal one, either both $p$ and $q$ are odd, or $p \equiv 0 \mod 4$ and $q$ is odd.  

If $d = 1$, $q \equiv 0 \mod 4$, and $p$ is odd, then we consider the following transformation.  Recall that the \emph{dual} $X^{\vee}$ of a Platonic surface $X$ with Schl\"afli symbol $\{p,q\}$ is the surface resulting from replacing every vertex by a face with $q$-sides and every face by a vertex with $p$-edges incident to it.\footnote{This is simply the generalization of the dual of a Platonic solid from antiquity.}  Consequentially, the Schl\"afli symbol of $X^{\vee}$ is $\{q,p\}$.  We claim that $\text{Rot}(X^{\vee}) \cong \text{Rot}(X)$, which will complete the proof by reducing to the case above.  The choice of pairs $(f,e)$ above for the rotation group was arbitrary and $(e,v)$, i.e. edge -- vertex, could have been chosen as well.  However, under this choice, $(e,v)$ on $X$ would become $(f,e)$ on $X^{\vee}$.  Since every element of each rotation group is completely determined by a pair of either edge-vertex pairs, or face-edge pairs, the bijection between such pairs induces an isomorphism between the groups.
\end{proof}

\section{Examples}
\label{ExampleSect}

We consider two examples.  The first of which was carried out almost to completion in \cite{AthreyaAulicinoHooperDodecahedron} and the second is a well-known example from genus two.


\subsubsection*{Auxiliary File}

For the convenience of the reader a Jupyter notebook titled \verb|code_from_section_5.ipynb| containing all of the code in this section is available for download on the arXiv with a preprint of this paper as well as on the author's website.  The reader may also find the website \cite{DodecahedronWebsite} helpful.  It contains examples of all of the classical Platonic solids.

\subsection{The Dodecahedron}

Here we present a new derivation of the rotation group of the dodecahedron using the calculations from \cite{AthreyaAulicinoHooperDodecahedron} and the theory developed above.  Recall the Schl\"afli symbol for the dodecahedron is $\{5,3\}$.  Therefore, Theorem \ref{GCD1Thm} applies.

The generators of the monodromy group were produced in \cite[$\S$ 6.3]{AthreyaAulicinoHooperDodecahedron} and running the code results in four generators.  However, only two are actually needed to generate the full group.

\begin{verbatim}
gen1_dodec = [(0, 19, 21), (1, 18, 4), (2, 56, 28), (3, 7, 51), 
(5, 49, 58), (6, 14, 16), (8, 17, 46), (9, 11, 12), (10, 44, 48), 
(13, 22, 36), (15, 39, 43), (20, 34, 38), (23, 27, 31), 
(24, 29, 26), (25, 54, 33), (30, 32, 59), (35, 55, 57), 
(37, 53, 40), (41, 50, 52), (42, 47, 45)]

gen2_dodec = [(0, 13, 16), (1, 12, 3), (2, 50, 46), (4, 25, 48), 
(5, 43, 9), (6, 8, 11), (7, 45, 36), (10, 38, 14), (15, 33, 19), 
(17, 40, 31), (18, 28, 26), (20, 58, 29), (21, 24, 23), 
(22, 30, 56), (27, 55, 51), (32, 53, 34), (35, 49, 52), 
(37, 47, 39), (41, 44, 42), (54, 57, 59)]
\end{verbatim}

Applying the Sage functions below show that these generators result in a group of order $60$.

\begin{verbatim}
GMon_dodec = PermutationGroup([gen1_dodec, gen2_dodec])
GMon_dodec.order()
\end{verbatim}

Therefore, this construction produces the rotation group of the dodecahedron as a faithful representation into $S_{60}$.

\subsection{The Bolza Surface}
\label{BolzaSurfSect}

In this example, we use Theorem \ref{GCD1Thm} to represent the rotation group of the Bolza surface as a subgroup of $S_{48}$.  By inspection of the universal covering of the Bolza surface, it has a fundamental domain that can be tiled by hyperbolic octagons, each consisting of $16$ hyperbolic triangles that all meet at a single vertex.  Each vertex of this octagonal tiling of the Bolza surface is incident with three octagons.  This decomposes the Bolza surface itself into six hyperbolic octagons.  By flattening each of the octagons, we get the flat Bolza surface depicted in Figure \ref{BolzaSurface}.  Since the Bolza surface is a Platonic surface with Schl\"afli symbol $\{8,3\}$, Theorem \ref{GCD1Thm} applies to this case to represent the rotation group as a subgroup of $S_{48}$ by computing the monodromy group.

We follow the computations in \cite[$\S$6, $\S$7]{AthreyaAulicinoHooperDodecahedron}.  In particular, the cube is most relevant here because, like the cube, the Bolza surface consists of even-sided polygonal faces.  Therefore, it avoids the extra computations needed for doubled odd polygons.

\begin{figure}
\begin{center}
\includegraphics[width=.8\textwidth]{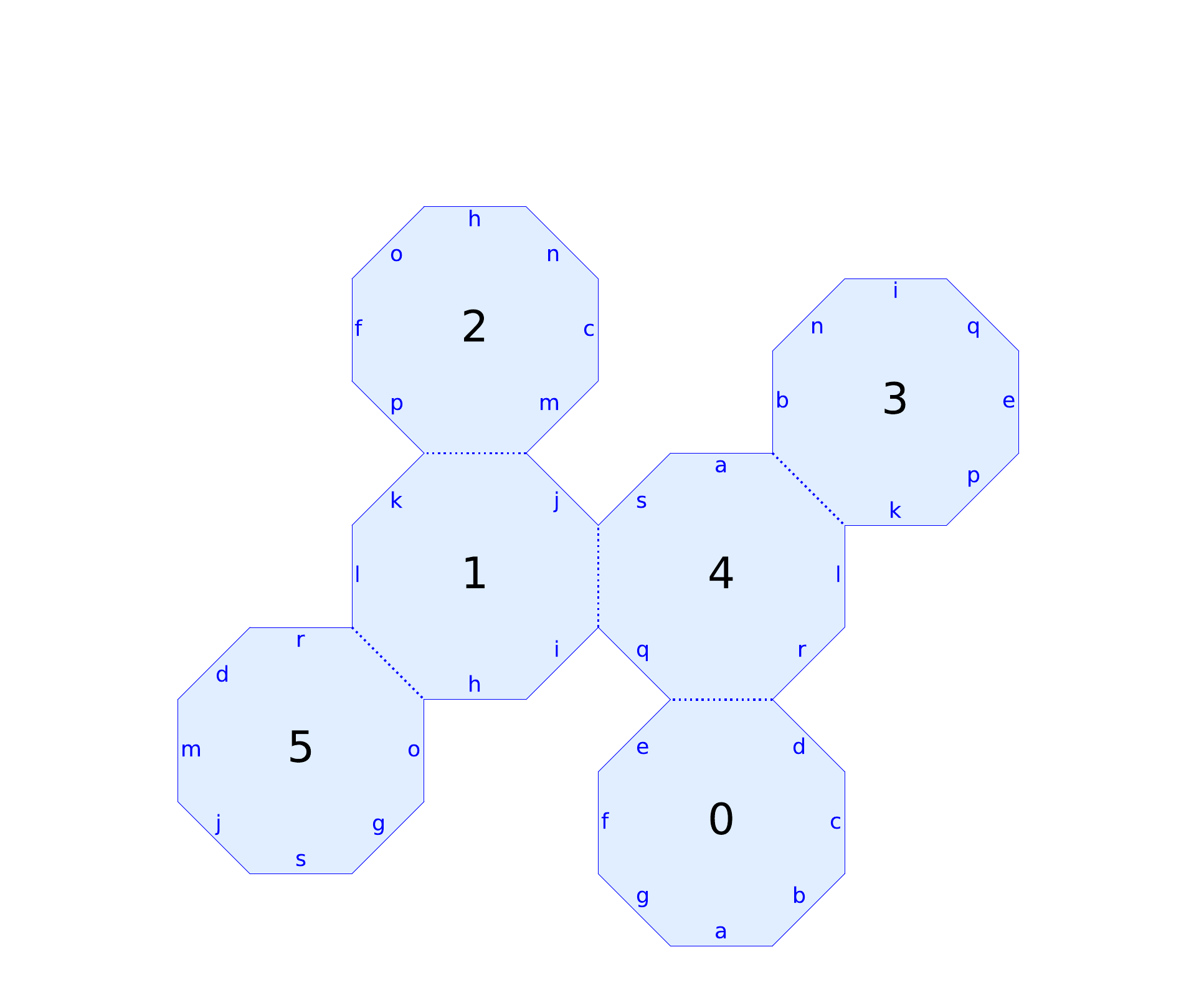}
\end{center}
\caption{The Bolza Surface}
\label{BolzaSurface}
\end{figure}

We follow the conventions established in \cite{AthreyaAulicinoHooperDodecahedron}.  Rather than labeling the octagons in Figure \ref{BolzaSurfaceUnfold} by individual numbers, it is actually much easier to label them by pairs $(i,j)$.  The coordinate $j \in \{0, \ldots, 5\}$ specifies the octagon seen in Figure \ref{BolzaSurface}.  The coordinate $i \in \{0, \ldots, 7\}$ specifies on which sheet in Figure \ref{BolzaSurfaceUnfold} the octagon lies.  Recall from \cite{AthreyaAulicinoHooperDodecahedron} that the advantage to this system is that the second coordinate is fixed as can be seen in the function \verb|build_adj_8_3|.

The function \verb|build_adj_8_3| assigns to each octagon a list of eight octagons to which it is adjacent using the convention that the bottom horizontal edge of each octagon in $\tilde D$ is labeled $0$, the next edge is $1$ going counter-clockwise, and continuing to $7$.  The lists are ordered so that index $i$ in the list gives the coordinates of the octagon incident with edge $i$ of octagon \verb|(sheet, octagon)|.  A particular symmetry of this surface that we exploit is that the parallel edges of every octagon are incident with the same octagon.  Therefore it suffices to double each list with the \verb|*2| command rather than repeat them.

\begin{verbatim}
def build_adj_8_3(sheet, octagon):
    i = sheet;
    oct_8_3_adj_base = 6*[None]
    oct_8_3_adj_base[0] = [[i,4],[i-1,3],[i-4,2],[i+2,5]]*2
    oct_8_3_adj_base[1] = [[i,2],[i+1,3],[i,4],[i,5]]*2
    oct_8_3_adj_base[2] = [[i,1],[i-1,5],[i-4,0],[i+2,3]]*2
    oct_8_3_adj_base[3] = [[i-1,1],[i-2,2],[i+1,0],[i,4]]*2
    oct_8_3_adj_base[4] = [[i,0],[i+1,5],[i,1],[i,3]]*2
    oct_8_3_adj_base[5] = [[i-1,4],[i-2,0],[i+1,2],[i,1]]*2
    prelim_adj = [oct_8_3_adj_base[octagon%6][(k-i)%8] for k in range(8)]
    return [[item[0]%8, item[1]%6] for item in prelim_adj]
\end{verbatim}

As in the case of the other Platonic solids from \cite[$\S$ 6]{AthreyaAulicinoHooperDodecahedron}, the only non-trivial line here is the line defining \verb|prelim_adj| because the index \verb|(k-i)%8| is not obvious.  However, this follows from the fact that the sheets of $\tilde D$ are numbered counter-clockwise in the presentation in Figure \ref{BolzaSurfaceUnfold}.  Therefore, $i$ rotations by $2\pi/8$ move edge $-i \mod 8$ into the lower horizontal position, whereby it becomes edge $0$.  Thus, the index follows.

Next we implicitly number all $48$ octagons in the cover by defining a complete list of all of them.

\begin{verbatim}
def octagons():
    return list(itertools.product(*[range(8), range(6)]))
\end{verbatim}

Finally, the permutations can be determined using the code that is a modification of the code for the cube from \cite[Ex. 6.1]{AthreyaAulicinoHooperDodecahedron}.  We again refer the reader to that paper for an explanation.

\begin{verbatim}
def perm_oct(abcd):
    oct_list = octagons()
    total = []
    i = 0
    perm_a_sub = []
    perm_a = []
    while len(total) < 46:
        total += perm_a_sub
        total.sort()
        if len(total) != 0:
            i_list = [j for j in range(len(total)) if j != total[j]]
            if i_list == []:
                i = len(total)
            else:
                i = i_list[0]
        perm_a_sub = []
        while i not in perm_a_sub:
            perm_a_sub += [i]
            i = oct_list.index(tuple(build_adj_8_3 \\
                (oct_list[i][0], oct_list[i][1])[abcd]))
        perm_a.append(tuple(perm_a_sub))
    return perm_a
\end{verbatim}

Given the functions above, we follow the commands for the dodecahedron above to produce the rotation group as a subgroup of $S_{48}$.

\begin{verbatim}
GMon_8_3 = PermutationGroup([perm_oct(0), perm_oct(1), \\
    perm_oct(2), perm_oct(3)])
GMon_8_3.order()
\end{verbatim}

The output of the last line is $48$, which is equal to the order of the rotation group by Theorem \ref{GCD1Thm}.  Hence, this gives a presentation of the rotation group of the Bolza surface as a subgroup of $S_{48}$.  Its rotation group is generated by the permutations \verb|perm_oct(0)|, \verb|perm_oct(1)|, \verb|perm_oct(2)|, and \verb|perm_oct(3)|.

\begin{remark}
We observe that the first three generators above suffice to generate the rotation group.
\end{remark}

\begin{figure}
\begin{center}
\includegraphics[width=\textwidth]{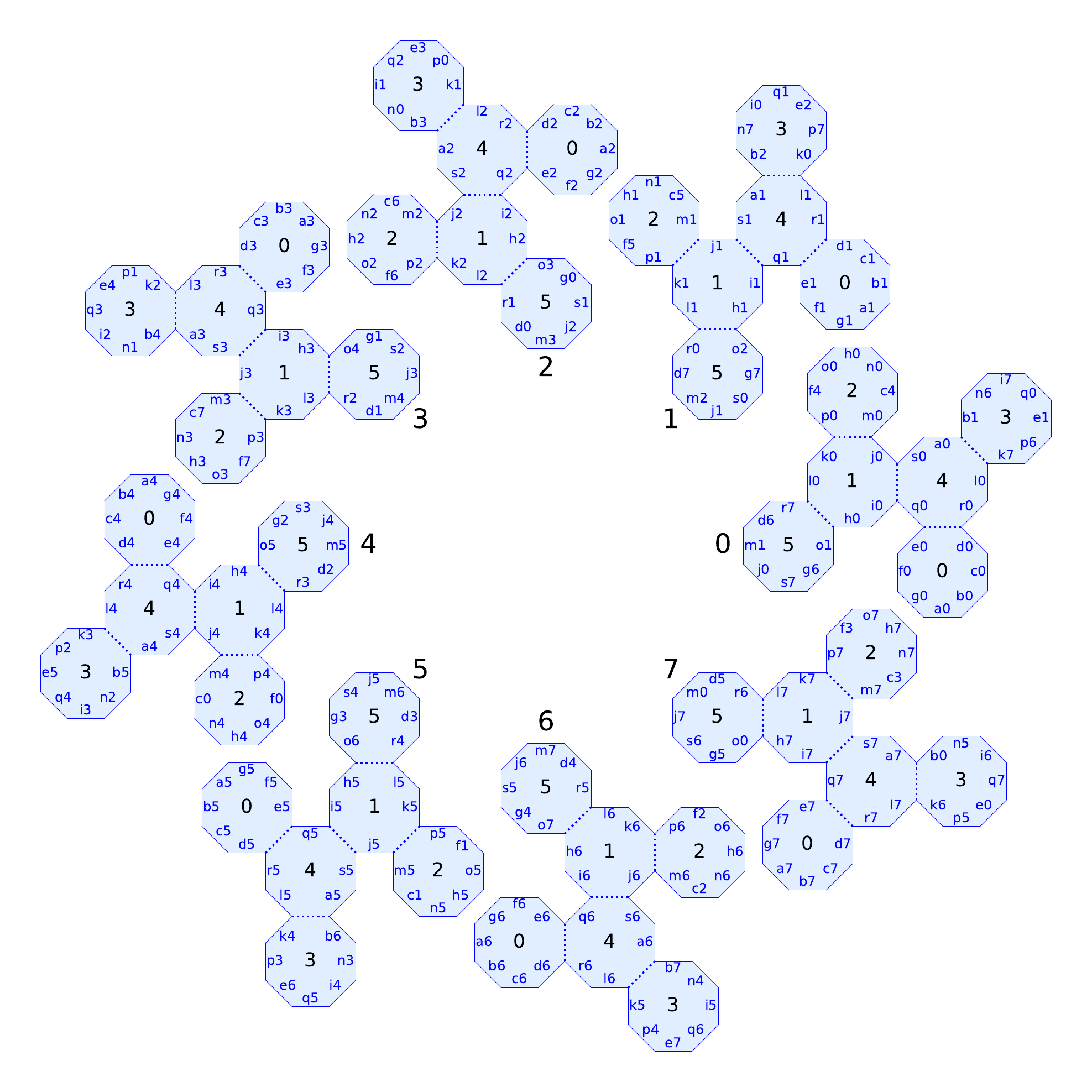}
\end{center}
\caption{The Unfolding of the Bolza Surface}
\label{BolzaSurfaceUnfold}
\end{figure}


\bibliography{fullbibliotex}{}

\end{document}